\documentclass[reqno,12pt]{amsart}

\usepackage{amsmath, amsthm, amsfonts, amssymb}
\input{mathrsfs.sty}

\usepackage{amsmath}
\usepackage{amsfonts}
\usepackage{amssymb}
\usepackage{amsthm}
\usepackage{amstext}

   \theoremstyle{plain}
   \newtheorem{thm}{Theorem}[section]
   \newtheorem{prop}[thm]{Proposition}
   \newtheorem{lem}[thm]{Lemma}
   
   \theoremstyle{definition}
   
   \newtheorem{defn}[thm]{Definition}
   \newtheorem{example}[thm]{Example}
   \theoremstyle{remark}

 \numberwithin{equation}{section}

\author{V. Manuilov}

\date{}

\address{Moscow Center for Fundamental and Applied Mathematics, Moscow State University,
Leninskie Gory 1, Moscow, 
119991, Russia}

\email{manuilov@mech.math.msu.su}

\thanks{The research was supported by RSF (project No. 21-11-00080).}

\title{On the inverse semigroup of bimodules over a $C^*$-algebra}

\begin{document}

\begin{abstract}
It was noticed recently that, given a metric space $(X,d_X)$, the equivalence classes of metrics on the disjoint union of the two copies of $X$ coinciding with $d_X$ on each copy form an inverse semigroup $M(X)$ with respect to concatenation of metrics. Now put this inverse semigroup construction in a more general context, namely, we define, for a $C^*$-algebra $A$, an inverse semigroup $S(A)$ of Hilbert $C^*$-$A$-$A$-bimodules. When $A$ is the uniform Roe algebra $C^*_u(X)$ of a metric space $X$, we construct a map $M(X)\to S(C^*_u(X))$ and show that this map is injective, but not surjective in general. This allows to define an analog of the inverse semigroup $M(X)$ that does not depend on the choice of a metric on $X$ within its coarse equivalence class.

\end{abstract}

\maketitle

\section*{Introduction}

It was noticed in \cite{M1} that, given a metric space $(X,d_X)$, the equivalence classes of metrics on the disjoint union of the two copies of $X$ coinciding with $d_X$ on each copy form an inverse semigroup $M(X)$ with respect to concatenation of metrics. It was found that $M(X)$ depends on a choice of a metric on $X$ within the coarse equivalence class of $d_X$. 

Now we are able to put this inverse semigroup construction in a more general context, namely, we define, for a $C^*$-algebra $A$, an inverse semigroup $S(A)$ of Hilbert $C^*$-$A$-$A$-bimodules. $C^*$-bimodules used to define this inverse semigroup are $A$-$A$-imprimitivity bimodules \cite{Rieffel} without the requirement that they should be full. 

When $A$ is the uniform Roe algebra $C^*_u(X)$ of a metric space $X$, we construct a map $M(X)\to S(C^*_u(X))$ and show that this map is injective, but not surjective in general. This allows to define an analog of the inverse semigroup $M(X)$ that does not depend on the choice of a metric on $X$ within its coarse equivalence class. 

Our references for inverse semigroup theory are \cite{Howie}, \cite{Lawson}. Details on coarse geometry and Roe algebras can be found in \cite{Roe}.

\section{Associative Hilbert $C^*$-bimodules as morphisms. The semigroup $S(A)$}

Let $A$, $B$ be $C^*$-algebras. 

\begin{defn}
A Banach space $M$ is an \textit{associative Hilbert $C^*$-bimodule (left over $A$ and right over $B$)} if
\begin{enumerate}
\item
$M$ is a left $A$-module and a right $B$-module, and $(am)b=a(mb)$ for any $a\in A$, $b\in B$, $m\in M$, and both modules are non-degenerate, i.e. $AM$ and $MB$ are dense in $M$;
\item
$M$ is a left Hilbert $C^*$-module over $A$ and a right Hilbert $C^*$-module over $B$ with respect to the inner products  ${}_A\langle \cdot,\cdot\rangle$ and $\langle \cdot,\cdot\rangle_B$, respectively; 
\item
$A$ (resp., $B$) acts on $M$ by bounded adjointable operators with respect to the Hilbert $C^*$-module structure over $B$ (resp., over $A$), i.e.  
$\langle am,m'\rangle_B=\langle m,a^*m'\rangle_B$ and ${}_A\langle mb,m'\rangle={}_A\langle m,m'b^*\rangle$ for any $a\in A$, $b\in B$, $m,m'\in M$;
\item
the norm on $M$ is equivalent to the norms $\|\cdot\|_A=\|{}_A\langle\cdot,\cdot\rangle\|^{1/2}$ and $\|\cdot\|_B=\|\langle\cdot,\cdot\rangle_B\|^{1/2}$;
\item
${}_A\langle m,n\rangle r=m\langle n,r\rangle_B$ for any $m,n,r\in M$.

\end{enumerate}

\end{defn}

We denote the set of associative Hilbert $C^*$-bimodules for fixed $A$ and $B$ by $\mathcal S(A,B)$. When $B=A$, we write $\mathcal S(A)$. For the set of isomorphism classes of associative Hilbert $C^*$-bimodules we write $S(A,B)$ and $S(A)$ respectively.

Let $J_A={}_A\langle M,M\rangle\subset A$ (resp., $J_B=\langle M,M\rangle_B\subset B$) be the closure of the linear span of ${}_A\langle m,n\rangle$ (resp., of $\langle m,n\rangle_B$), $m,n\in M$. These are ideals in $A$ and $B$, respectively (by ideals we mean closed two-sided ideals). The definition of $A$-$B$-imprimitivity bimodules \cite{Rieffel} requires, besides the items 1-5 above, that $M$ is a \textit{full} Hilbert $C^*$-bimodule, i.e. that $J_A=A$ and $J_B=B$. Then such bimodules are used to define strong Morita equivalence, but we don't assume that $M$ is full. Note that we consider $M\in\mathcal S(A,B)$ as a left $J_A$-module and a right $J_B$-module. Then $M$ is a $J_A$-$J_B$-imprimitivity bimodule (and $J_A$ and $J_B$ are strongly Morita equivalent). 

Associative Hilbert $C^*$-bimodules can be considered as morphisms on the category of $C^*$-algebras. If $C$ is a $C^*$-algebra and $N\in\mathcal S(B,C)$ then $M\otimes_B N$ is an associative Hilbert $C^*$-bimodule, left over $A$ and right over $C$, with respect to the inner products determined by 
$$
{}_A\langle m\otimes n,m'\otimes n'\rangle={}_A\langle m,m'\langle n,n'\rangle_B\rangle
$$ 
and 
$$
\langle m\otimes n,m'\otimes n'\rangle_C=\langle\langle m,m'\rangle_B^* n,n'\rangle_C, 
$$
$m,m'\in M$, $n,n'\in N$. We write $M\cdot N$ for $M\otimes_B N$. This makes $S(A)$ a semigroup.

For $M\in\mathcal S(A,B)$, the dual module $M^*\in\mathcal S(B,A)$ was defined in \cite{Rieffel}, Definition 6.17 (it was denoted there by $\widetilde{M}$). Recall that $M^*$ as a set is the same as $M$, while the bimodule structure is given by 
$$
b\tilde m=(mb^*)\tilde{\phantom{a}}, \quad \tilde{m}a=(a^*m)\tilde{\phantom{a}}, \quad {}_B\langle \tilde m,\tilde n\rangle=\langle m,n\rangle_B, \quad \langle\tilde m,\tilde n\rangle_A={}_A\langle m,n\rangle, 
$$
where we write $\tilde m$ for $m\in M$ considered as an element of $M^*$. When $B=A$ we say that $M$ is selfadjoint if $M^*$ is isomorphic to $M$ as an associative Hilbert $C^*$-bimodule.

\begin{example}
Let $B=A$, $J\subset A$ an ideal. Then $J\in\mathcal S(A)$. In particular, $J=\{0\}$ (resp., $J=A$) represents the zero (resp., the unit) element in the semigroup $S(A)$. 

\end{example}

\begin{lem}\label{Rieffel-lemma_6.22}
Let $M\in\mathcal S(A,B)$. Then $M^*\cdot M\cong J_B$, where $J_B=\langle M,M\rangle_B$.

\end{lem}
\begin{proof}
It was shown in Lemma 6.22 of \cite{Rieffel} that the map $R:M^*\otimes_{J_A}M\to J_B$ defined by $R(\tilde m\otimes n)=\langle m,n\rangle_B$ is an isomorphism. As $M^*\cdot M=M^*\otimes_A M$, it remains to check that $M^*\otimes_{J_A}M=M^*\otimes_A M$. 

Let $m\in M$, $a\in A$, $\varepsilon>0$, $(u_\lambda)_{\lambda\in\Lambda}$ is an approximate unit in the ideal $J_A$. Then 
\begin{eqnarray*}
{}_A\langle (a-u_\lambda a)m,(a-u_\lambda a)m \rangle&=&(a-u_\lambda a){}_A\langle m,m\rangle (a-u_\lambda a)^*\\
&\leq& \|a\|\cdot\|(a-u_\lambda a){}_A\langle m,m\rangle\|\\
&=&\|a\|\cdot\|a'-u_\lambda a'\|,
\end{eqnarray*}
where $a'=a{}_A\langle m,m\rangle\in J_A$. Therefore, there exists $j\in J_A$ of the form $j=u_\lambda a$ such that $\|am-jm\|<\varepsilon$. Then for any $m,n\in M$, any $a\in A$ and any $\varepsilon>0$ there exists $j\in J$ such that 
$$
\|(\tilde m a\otimes n-\tilde m\otimes an)-(\tilde m j\otimes n-\tilde m\otimes jn)\|<\varepsilon, 
$$
hence $\|\tilde m a\otimes n-\tilde m\otimes an\|=0$ in $M^*\otimes_{J_A}M$. As the kernel of the canonical quotient map $M^*\otimes_{J_A}M\to M^*\otimes_A M$ is generated by differences $\tilde m a\otimes n-\tilde m\otimes an$, $a\in A$, $m,n\in M$, this kernel is trivial.

\end{proof}

\begin{defn}
An associative Hilbert $C^*$-bimodule $M\in\mathcal S(A)$ is \textit{idempotent} if $M\cdot M\cong M$. A selfadjoint idempotent is a \textit{projection}.

\end{defn}  

\begin{example}
Let $J\subset A$ be an ideal. Consider it as a Hilbert $C^*$-bimodule, $J\in\mathcal S(A)$. Then $J$ is a projection. Clearly, it is selfadjoint. As $J_A=J$, by Lemma \ref{Rieffel-lemma_6.22}, $J^*\otimes_A J\cong J$. 

\end{example}

\begin{lem}
If $M\in\mathcal S(A)$ is a projection then there exists an ideal $J\subset A$ such that $M\cong J$.

\end{lem}
\begin{proof}
If $M$ is a projection then $M\cdot M\cong M^*\cdot M\cong M$. Let $\langle M,M\rangle_A=J\subset A$. Then $M\cong M^*\cdot M\cong J$. 

\end{proof}

\begin{lem}
The semigroup $S(A)$ is regular.

\end{lem}
\begin{proof}
Recall that a semigroup $S$ is regular if any element $s\in S$ has a `pseudoinverse' $t\in S$ such that $sts=s$ and $tst=t$. This follows from the isomorphism $M\cdot M^*\cdot M\cong M$ for any $M\in\mathcal S(A)$, so let us check this isomorphism. Let $\langle M,M\rangle_A=J_A\subset A$. Then $M^*\cdot M\cong J_A$, hence $M\cdot M^*\cdot M\cong M\otimes_A J_A$. 
As in the proof of Lemma \ref{Rieffel-lemma_6.22}, using an approximate unit in $J_A$, it is easy to show that the canonical surjection $M=M\otimes_{J_A}J_A\to M\otimes_A J_A$ is an isomorphism.

\end{proof}

\begin{lem}\label{commute}
Let $I,J\subset A$ be ideals. Then $I\cdot J=J\cdot I\cong I\cap J$.

\end{lem}
\begin{proof}
The map $R_J:I\otimes_A J\to I$ is an isometry, hence it remains to show that the range of $R_J$ is $I\cap J$. This can be done by using that any element $a\in I\cap J$ can be written as a product $a=ij$, where $i\in I$, $j\in J$. Similarly, $J\otimes_A I\cong I\cap J$. 

\end{proof}

\begin{thm}
The semigroup $S(A)$ is an inverse semigroup.

\end{thm}
\begin{proof}
Recall that a regular semigroup $S$ is an inverse semigroup iff any two idempotents commute.
By Lemma \ref{commute}, any two projections commute. This suffices: if $M$ is idempotent then $M^*$ is idempotent as well, and
\begin{eqnarray*}
M^*&=&M^*\cdot M\cdot M^*=(M^*\cdot M)\cdot (M\cdot M^*)\\
&=&(M\cdot M^*)\cdot (M^*\cdot M)=M\cdot M^*\cdot M\\
&=&M,
\end{eqnarray*}
hence $M$ is a projection.

\end{proof}


\begin{prop}
Let $A$ and $B$ be stronly Morita equivalent $C^*$-algebras. Then $S(A)\cong S(B)$.

\end{prop}
\begin{proof}
Let $M\in \mathcal S(A,B)$ be an imprimitivity $A$-$B$-bimodule, i.e. a full associative Hilbert $C^*$-bimodule. Then the maps $P\mapsto M\otimes_B P$ and $Q\mapsto M^*\otimes_A Q$, where $P\in\mathcal S(B)$, $Q\in\mathcal S(A)$, are inverse to each other. 

\end{proof}

\begin{prop}
Let $J\subset A$ be an ideal. Then $S(J)\subset S(A)$.

\end{prop}
\begin{proof}
Let $M\in\mathcal S(J)$. We show that $M$ can be considered as an $A$-$A$-bimodule. Let $m\in M$, $a\in A$, and let $\{u_\lambda\}_{\lambda\in\Lambda}$ be an approximate unit in $J$. Then set $a\cdot m=\lim_{\Lambda}(au_\lambda)m$. Existence of the limit follows from the estimate 
\begin{eqnarray*}
\|(au_\lambda)m-(au_\mu)m\|^2&=&\|\langle(a(u_\lambda-u_\mu))m,(a(u_\lambda-u_\mu))m\rangle_A\|\\
&\leq&\|a^*(u_\lambda-u_\mu)\langle m,m\rangle_A(u_\lambda-u_\mu)a\| 
\end{eqnarray*}
and from convergence of $(u_\lambda-u_\mu)\langle m,m\rangle_A$.  

\end{proof}

Recall that the isomorphism classes of imprimitivity bimodules over $A$ form a group, called the Picard group of $A$ \cite{Brown-Green-Rieffel}.

\begin{prop}
If $A$ is simple then $S(A)=\{0\}\cup\operatorname{Pic}(A)$.

\end{prop}
\begin{proof}
Let $M\in\mathcal S(A)$, $M\neq 0$. As $M^*\cdot M$ and $M\cdot M^*$ are ideals in $A$, we have $M^*\cdot M\cong A\cong M\cdot M^*$, and these isomorphisms are given by $S:\tilde m\otimes n\mapsto \langle m,n\rangle_A$ and $T:m\otimes\tilde n\mapsto{_A\langle m,n\rangle}$, hence $M$ is a full Hilbert $C^*$-nodule, i.e. an imprimitivity bimodule.   

\end{proof}

\begin{prop}
$S(\mathbb C^n)$ is the semigroup of partial bijections of the set $\{1,\ldots,n\}$.

\end{prop}
\begin{proof}
Let $X=\{1,\ldots,n\}$, $\mathbb C^n=C(X)=A$, and let $M\in\mathcal S(C(X))$. Let $J_1={}_A\langle M,M\rangle$, $J_2=\langle M,M\rangle_A$. Then $J_1=C(P)$, $J_2=C(Q)$, where $P,Q\subset X$, and $M$ is $J_1$-$J_2$-imprimitivity bimodule. $C(P)$ and $C(Q)$ are strongly Morita equivalent only if $|P|=|Q|$. The Picard group of $C(P)$ is the group of permutations of $P$, so $M$ determines and is determined by a partial bijection $P\to Q$ (cf. \cite{Brown-Green-Rieffel}).  

\end{proof}

\section{Inverse semigroup from metrics on doubles and Roe bimodules}

Let $X=(X,d_X)$ be a countable discrete metric space, and let $H_X=l^2(X)$ denote the Hilbert space of square-summable complex-valued functions on $X$, with the orthonormal basis consisting of delta functions $\delta_x$, $x\in X$, and with the inner product $(\cdot,\cdot)$. An operator $T\in\mathbb B(H_X)$ is said to have propagation $\leq L$ if $T_{x,y}=0$ whenever $d_X(x,y)>L$, where $T_{x,y}=(\delta_y,T\delta_x)$. The norm closure of the set of all operators of finite propagation is the uniform Roe algebra $C^*_u(X)$.

For a countable discrete metric space $X=(X,d_X)$, let $\mathcal M(X)$ denote the set of metrics on $X\times\{0,1\}$ such that
\begin{itemize}
\item
for any $d\in\mathcal M(X)$, the restriction of $d$ to each copy of $X$ equals $d_X$;
\item
the distance between the two copies of $X$ is non-zero.
\end{itemize}
Let $M(X)$ be the set of coarse equivalence classes of metrics in $\mathcal M(X)$. It was shown in \cite{M1} that $M(X)$ is an inverse semigroup with respect to the concatenation of metrics.

Let $(X,d_X)$ and $(Y,d_Y)$ be two countable discrete metric spaces. Let $Z=X\sqcup Y$, and let $D_{X,Y}$ denote the set of all metrics $d$ on $Z$ such that $d|_X=d_X$ and $d|_Y=d_Y$. For each $d\in D(X,Y)$, let $\mathbb M_d[X,Y]$ denote the set of all bounded finite propagation operators $T:H_X\to H_Y$, and let $M_d(X,Y)$ be its norm closure in the bimodule $\mathbb B(H_X,H_Y)$ of all bounded operators from $H_X$ to $H_Y$. 

If $T,S\in \mathbb M_d[X,Y]$ then $T^*S$ and $TS^*$ are finite propagation operators in $l^2(X)$ and $l^2(Y)$, respectively, hence we can define uniform Roe algebras-valued inner products 
$$
{}_{C^*_u(Y)}\langle T,S\rangle=TS^*\quad\mbox{and}\quad \langle T,S\rangle_{C^*_u(X)}=T^*S. 
$$
Similarly, $M_d(X,Y)$ is a left $C^*_u(Y)$-module and a right $C^*_u(X)$-module. Clearly, this bimodule is associative. We write $M_d(X)$ when $Y=X$.

\begin{lem}\label{isom1}
Let $d_1,d_2\in\mathcal M(X,Y)$, $M_j=M_{d_j}(X,Y)$, $j=1,2$, and let $M_1\cong M_2$.
Let $f:M_1\to M_2$ be an isomorphism of associative Hilbert $C^*$-bimodules over $C^*_u(X)$ and $C^*_u(Y)$. Then $f(m)=\lambda m$ for any $m\in M_1$, where $\lambda\in\mathbb C$, $|\lambda|=1$. In particular, $M_1=M_2$.

\end{lem}
\begin{proof}
Let $x\in X$, $y\in Y$, and let $e_{x,y}$ denote the elementary operator corresponding to these two points, i.e. $e_{x,y}\delta_z=\left\lbrace\begin{array}{cl}\delta_y&\mbox{if\ }z=x;\\0&\mbox{if\ }z\neq x,\end{array}\right.$ where $\delta_x\in H_X$ is the delta-function of the point $x$. Clearly, $e_{x,y}\in M_d(X,Y)$ for any $d\in \mathcal M(X,Y)$. The elementary operators $e_{x,x}$ and $e_{y,y}$ lie in $C^*_u(X)$ and in $C^*_u(Y)$, respectively. Then
$$
f(e_{x,y})=f(e_{y,y}e_{x,y}e_{x,x})=e_{y,y}f(e_{x,y})e_{x,x}=\lambda_{x,y}e_{x,y} 
$$ 
for some $\lambda_{x,y}\in\mathbb C$, and as $f$ is an isometry, $|\lambda_{x,y}|=1$.

Let $z\in X$. Then $e_{z,x}\in C^*_u(X)$ and $e_{z,y}=e_{x,y}e_{z,x}$, so
$$
\lambda_{z,y}e_{z,y}=f(e_{z,y})=f(e_{x,y}e_{z,x})=f(e_{x,y})e_{z,x}=\lambda_{x,y}e_{x,y}e_{z,x}=\lambda_{x,y}e_{z,y},
$$
hence $\lambda_{z,y}=\lambda_{z,y}$ for any $x,z\in X$ and $y\in Y$. Similarly, $\lambda_{x,y}=\lambda_{x,u}$ for any $y,u\in Y$ and $x\in X$. Thus, $\lambda_{x,y}=\lambda$ for any $x\in X$, $y\in Y$.

Let $m\in M_1$, $m=\sum_{x\in X,y\in Y}m_{x,y}e_{x,y}$. Then 
\begin{eqnarray*}
f(m)_{x,y}e_{x,y}&=&e_{y,y}f(m)e_{x,x}=f(e_{y,y}me_{x,x})=f(e_{y,y}m_{x,y}e_{x,y}e_{x,x})=
m_{x,y} f(e_{x,y})\\
&=&\lambda m_{x,y} e_{x,y}, 
\end{eqnarray*} 
hence $f(m)=\lambda m$.

\end{proof}

This gives a semigroup homomorphism 
\begin{equation}\label{map1}
i:M(X)\to S(C^*_u(X)),\quad i([d])=[M_d(X)].
\end{equation}

\begin{thm}
The map $i$ (\ref{map1}) is injective.

\end{thm}
\begin{proof}
Let $d_1,d_2\in\mathcal M(X)$, $M_j=M_{d_j}(X)\in i([d_j])$, $j=1,2$, and let $M_1\cong M_2$. 
Assume the contrary, i.e. that $[d_1]\neq [d_2]$. Then there exist two sequences $(x_n),(y_n)$, $n\in\mathbb N$, of points in $X$, and $C>0$ such that $d_1((x_n,0),(y_n,1))<C$ for any $n\in\mathbb N$, and $\lim_{n\to\infty}d_2((x_n,0),(y_n,1))=\infty$.
Set $m=\sum_{n\in\mathbb N}e_{(x_n,0),(y_n,1)}\in\mathbb B(H_X,H_Y)$ (the sum is strongly convergent). Then $m\in M_1$, but $m\notin M_2$, but this conradicts $M_1\cong M_2$ by Lemma \ref{isom1}.

\end{proof}

On the other hand, the map $i$ (\ref{map1}) is far from being surjective. We show this for two examples, one very simple, and another more deep.

\begin{example}
Let $X$ be a countable set with the metric $d_X$ defined by $d_X(x,y)=1$ if $x\neq y$, $x,y\in X$. Then $C_u^*(X)=\mathbb B(H_X)$. Any two metrics in $\mathcal M(X)$ are coarsely equivalent, hence $M(X)$ consists of a single element. On the other hand, there are at least two different associative Hilbert $C^*$-bimodules over $\mathbb B(H_X)$: $\mathbb B(H_X)$ itself and the ideal $\mathbb K(H_X)$ of compact operators. It is easy to see that $i(M(X))=[\mathbb B(H_X)]$.

\end{example} 

Before we turn to a more deep example, let us recall that a discrete metric space $X$ is of bounded geometry if the numbers of points in all balls of radius $R$ are uniformly bounded for any $R>0$. 

An element $a\in C^*_u(X)$ is a ghost element if $\lim_{x,z\to\infty}a_{x,y}=0$, where $a_{x,z}$, $x,z\in X$, denote the matrix entries of $a$ with respect to the basis $\delta_x$, $x\in X$. The set of all ghost elements forms an ideal in $C^*_u(X)$ called the ghost ideal. This ideal contains the ideal $\mathbb K(H_X)$ of all compact operators, and it was shown in \cite{Roe-Willett} that it equals $\mathbb K(H_X)$ if and only if $X$ has property A of Guoliang Yu.

\begin{prop}
Let $X$ be a discrete metric space of bounded geometry without property A. Then the map $i$ is not surjective.

\end{prop}
\begin{proof}
Let $I\subset C^*_u(X)$ denote the ghost ideal. Then it properly contains $\mathbb K(H_X)$. We claim that $[I]\in S(C^*_u(X))$ does not lie in the image of $i$. Assume the contrary: let $[d]\in\mathcal M(X)$, $i([d])=[I]$. As $[I]$ is an idempotent in $S(C^*_u(X))$, $[d]$ should be an idempotent in $M(X)$. The characterisation of idempotents was obtained in \cite{M2}: any such $d$ is coarsely equivalent to the following metric 
$$
d_{\mathcal A}((x,0),(y,1))=\inf_{n\in\mathbb N}\inf_{z\in A_n}d_X(x,z)+n+d_X(z,y),
$$ 
where $\{A_n\}_{n\in\mathbb N}$ is a sequence of subsets in $X$ such that, for each $n\in\mathbb N$, $A_{n+1}$ contains the 1-neighborhood $N_1(A_n)$ of $A_n$, i.e. the set $N_1(A_n)=\{x\in X:d_X(x,A_n)\leq 1\}$. 

The metric $d_{\mathcal A}$ represents the zero element $0\in M(X)$ if each $A_n$ is bounded. If [d]=0 then the bimodule $M_d(X)$ equals the set of compact operators, which differs from $[I]$, so $[d]\neq 0$. Then there exists $n\in\mathbb N$ such that $A_n$ is not bounded, i.e. there exists a sequence $\{x_k\}_{k\in\mathbb N}$ of points in $A_n$ such that $\lim_{k\to\infty}d_X(x_k,\{x_1,\ldots,x_{k-1}\})=\infty$. Then $d_{\mathcal A}((x_k,0),(x_k,1))\leq n$ for any $k\in\mathbb N$. As $d$ and $d_{\mathcal A}$ are coarsely equivalent, there exists $C>0$ such that $d((x_k,0),(x_k,1))\leq C$ for any $k\in\mathbb N$. Set $m=\sum_{k\in\mathbb N}e_{(x_k,0),(x_k,1)}$ (the sum is strongly convergent). Then $m\in M_d(X)$, but $m\notin I$, hence $i([d])=[M_d(X)]\neq [I]$ (as in the proof of Lemma \ref{isom1}, isomorphism of associative Hilbert $C^*$-bimodules represented on $H_X\oplus H_X$ implies equality of these bimodules). 

\end{proof}

Recall that two metrics, $d_X$ and $b_X$, on $X$, are coarsely equivalent if there exists a homeomorphism $\varphi$ of $[0,\infty)$ such that $\varphi^{-1}(b_X(x,y))<d_X(x,y)<\varphi(b_X(x,y))$ for any $x,y\in X$.
It was shown in \cite{M3} that the inverse semigroup $M(X)$ is not coarsely invariant, e.g. there exists a space $X$ with two coarsely equivalent metrics, $d_X$ and $d'_X$ such that $M(X,d_X)$ is commutative, while $M(X,d'_X)$ is not. Now we can define a new semigroup $M_c(X)$ for a metric space $(X,d_X)$ by 
$$
M_c(X)=\cap_{d\in[d_X]}i(M(X,d)).
$$ 
\begin{prop}
$M_c(X)$ is an inverse semigroup invariant under coarse equivalence of metrics on $X$.

\end{prop}

\begin{example}
Let $x_n=n^2$, $X=\{x_n\}_{n\in\mathbb Z}\subset\mathbb R$ equipped with the metric $d_X$ induced from the standard metric on $\mathbb R$. 
Let $Y=\{n^3:n\in\mathbb N\}$, and let $f:X\to Y$ be the map defined by $f(x_n)=\left\lbrace\begin{array}{cl}(2n)^3,&\mbox{if\ }n \geq 0;\\
-(2n+1)^3,&\mbox{if\ }n<0.\end{array}\right.$ Let $b_X$ be themetric on $X$ defined by $b_X(x_n,x_m)=|f(x_n)-f(x_m)|$, $n,m\in\mathbb Z$. It is easy to see that $X$ and $Y$ are coarsely equivalent, hence the metrics $d_X$ and $b_X$ are coarsely equivalent. While $M(X,d_X)$ is not commutative, $M(S,b_X)$ is commutative by Proposition 7.1 of \cite{M1}. Thus, $M_c(X)$ should be commutative. 

\end{example}

\end{document}